\documentclass[
    ,final            
  ]
  {aipproc}

\layoutstyle{6x9}

\usepackage{amssymb}

\newtheorem{definition}{Definition}
\newtheorem{lemma}[definition]{Lemma}
\newtheorem{theorem}[definition]{Theorem}
\newtheorem{proposition}[definition]{Proposition}

\newtheorem{remark}[definition]{Remark}
\newenvironment{proof}{\noindent{\bf Proof }}{\hfill $\Box$\medskip}

\begin{document}

\title{On non-formality of a simply-connected symplectic $8$-manifold}


\author{Gil R. Cavalcanti}{
  address={Mathematical Institute, St. Giles
24 -- 29, Oxford OX1 3LB, United Kingdom} }

\author{Marisa Fern\'andez}{
  address={Departamento de Matem\'aticas,
Facultad de Ciencia y Tecnolog\'{\i}a, Universidad del Pa\'{\i}s
Vasco, Apartado 644, 48080 Bilbao, Spain} }

\author{Vicente Mu\~noz}{
  address={Instituto de Ciencias Matem\'aticas CSIC-UAM-UCM-UC3M, Consejo
Superior de Investigaciones Cient{\'\i}ficas, C/ Serrano 113bis, 28006
Madrid, Spain}
  ,altaddress={Facultad de Matem{\'a}ticas, Universidad Complutense de Madrid, Plaza
de Ciencias 3, 28040 Madrid, Spain} 
}

\begin{abstract}
We show an alternative construction of the first example of a
simply-connected compact symplectic non-formal $8$-manifold given
in \cite{FM4}. We also give an alternative proof of its
non-formality using higher order Massey products.
\end{abstract}

\maketitle

\section{Introduction}

In \cite{ref1,ref2,ref3} Babenko--Taimanov and Rudyak--Tralle give
examples of non-formal simply-connected compact symplectic
manifolds of any even dimension bigger than or equal to $10$.
Babenko and Taimanov raise the question of the existence of
non-formal simply-connected compact symplectic manifolds of
dimension $8$, which cannot be constructed with their methods. In
\cite{FM4}, it is constructed the first example of a
simply-connected compact symplectic $8$-dimensional manifold which
is non-formal, thereby completing the solution to the question of
existence of non-formal symplectic manifolds for all allowable
dimensions. This example is constructed by starting with a
suitable complex $8$-dimensional compact nilmanifold $M$ which has
a symplectic form (but is not K\"ahler). Then one quotients by a
suitable action of the finite group ${\mathbb{Z}}_3$ acting
symplectically and freely except at finitely many fixed points.
This gives a symplectic orbifold $\widehat M=M/{\mathbb{Z}}_3$,
which is non-formal and simply-connected thanks to the choice of
${\mathbb{Z}}_3$-action. The last step is a process of symplectic
resolution of singularities to get a smooth symplectic manifold.
The symplectic resolution of isolated orbifold singularities has
been described in detail in \cite{CMF}. The non-formality of
$\widehat M$ is checked via a newly defined product in cohomology.
This is a product of Massey type, which is called $a$-product, and
it is discussed at length in \cite{CMF}.

The purpose of the present note is to give a new description of
the symplectic orbifold $\widehat M$ defined in \cite{FM4}. The
description presented here is in terms of real nilpotent Lie
groups. Secondly, we prove the non-formality of $\widehat M$ by
using higher order Massey products instead of $a$-products. It
remains thus open the question of the existence of a smooth
$8$-manifold with non-zero $a$-products but trivial (higher order)
Massey products.

\section{A nilmanifold of dimension $6$} \label{sec:preliminaries}

Let $G$ be the simply connected nilpotent Lie group of dimension
$6$ defined by the structure equations
 \begin{equation}\label{eqn:a}
 \begin{array}{lll}
  && d\beta_i=0, \qquad i=1,2 \\
  && d\gamma_i=0,\qquad i=1,2 \\
  && d\eta_1=-\beta_1\wedge \gamma_1+ \beta_2\wedge \gamma_1+
     \beta_1\wedge \gamma_2+ 2  \beta_2\wedge \gamma_2, \quad \\
  && d\eta_2=2 \beta_1\wedge \gamma_1+ \beta_2\wedge \gamma_1+
     \beta_1\wedge \gamma_2 - \beta_2\wedge \gamma_2,
 \end{array}
 \end{equation}
where $\{\beta_i,\gamma_i,\eta_i; 1\leq i \leq 2\}$ is a basis of
the left invariant $1$--forms on $G$. Because the structure
constants are rational numbers, Mal'cev theorem \cite{Malc}
implies the existence of a discrete subgroup $\Gamma$ of $G$ such
that the quotient space $N=\Gamma{\backslash} G$ is compact.

Using Nomizu's theorem \cite{No} we can compute the real
cohomology of $N$. We get
\begin{eqnarray*}
 H^0(N)& =& \langle 1\rangle,\\
 H^1(N) &=& \langle [\beta_1], [\beta_2],[\gamma_1],[\gamma_2]\rangle,\\
 H^2(N) &=& \langle [\beta_1 \wedge \beta_2], [\beta_1 \wedge
 \gamma_1],
  [\beta_1 \wedge \gamma_2], [\gamma_1\wedge \gamma_2],
 [\beta_1\wedge \eta_2 - \beta_2\wedge \eta_1],
  [\gamma_1\wedge \eta_2 - \gamma_2\wedge \eta_1],\\
  &&
 [\beta_1\wedge \eta_1 + \beta_1\wedge \eta_2+\beta_2\wedge \eta_2],
  [\gamma_1\wedge \eta_1 + \gamma_1\wedge \eta_2+\gamma_2\wedge \eta_2]\rangle,\\
  H^3(N) &=& \langle [\beta_1\wedge\beta_2\wedge \eta_1],
 [\beta_1\wedge\beta_2\wedge \eta_2],
 [\gamma_1\wedge\gamma_2\wedge \eta_1],
 [\gamma_1\wedge\gamma_2\wedge \eta_2],
 [\beta_1\wedge\gamma_1\wedge (\eta_1+2\eta_2)],\\
 & &
   [\beta_1 \wedge \gamma_1 \wedge \eta_2-\beta_1 \wedge \gamma_2 \wedge \eta_1],
   [\beta_1 \wedge \gamma_2 \wedge \eta_1-\beta_1 \wedge \gamma_2 \wedge \eta_2],
 [\beta_2 \wedge \gamma_2 \wedge (\eta_2+2\eta_1)],  \\
 & &
  [\beta_2 \wedge \gamma_2 \wedge \eta_1 -\beta_2 \wedge \gamma_1 \wedge \eta_2],
  [\beta_2 \wedge \gamma_1 \wedge \eta_2-\beta_2 \wedge \gamma_1 \wedge   \eta_1]\rangle,\\
 H^4(N) &=& \langle [\beta_1 \wedge \beta_2 \wedge \gamma_1 \wedge \eta_1],
   [\beta_1 \wedge \beta_2 \wedge \gamma_1 \wedge \eta_2],
  [\beta_1 \wedge \beta_2 \wedge \eta_1 \wedge \eta_2],
 [\beta_1 \wedge \gamma_1 \wedge \gamma_2 \wedge \eta_2],\\
  & &   [\beta_2 \wedge \gamma_1 \wedge \gamma_2 \wedge \eta_2],
   [\gamma_1 \wedge \gamma_2 \wedge \eta_1 \wedge \eta_2],
   [\beta_1 \wedge \gamma_2 \wedge \eta_1 \wedge \eta_2-
   \beta_2 \wedge \gamma_1 \wedge \eta_1 \wedge \eta_2],\\
  & & [\beta_1 \wedge \gamma_2 \wedge \eta_1 \wedge \eta_2+
 \beta_1 \wedge \gamma_1 \wedge \eta_1 \wedge \eta_2+
 \beta_2 \wedge \gamma_2 \wedge \eta_1 \wedge \eta_2]\rangle,\\
 H^5(N) &=& \langle [\beta_1 \wedge \beta_2 \wedge \gamma_1 \wedge
  \eta_1 \wedge \eta_2],
 [\beta_1 \wedge \beta_2 \wedge \gamma_2 \wedge
  \eta_1 \wedge \eta_2],
  [\beta_1 \wedge \gamma_1 \wedge \gamma_2 \wedge
  \eta_1 \wedge \eta_2],\\
 & &
  [\beta_2 \wedge \gamma_1 \wedge \gamma_2 \wedge
  \eta_1 \wedge \eta_2] \rangle, \\
  H^6(N) &=& \langle [\beta_1 \wedge \beta_2 \wedge \gamma_1 \wedge \gamma_2
  \wedge \eta_1 \wedge \eta_2]\rangle.
 \end{eqnarray*}

We can give a more explicit description of the group $G$. As a
differentiable manifold $G={\mathbb{R}}^6$. The nilpotent Lie
group structure of $G$ is given by the multiplication law
 \begin{equation}\label{eqn:m}
 \begin{array}{ccl}
  m:\qquad G \times G&\longrightarrow& \qquad G \\
  \left((y_1',y_2',z_1',z_2',v_1',v_2'),(y_1,y_2,z_1,z_2,v_1,v_2)\right)
  &\mapsto & \Big(y_1+y_1',y_2+y_2',z_1+z_1',z_2+z_2', \\ &&
  v_1 +v_1'+(y_1' -y_2') z_1 - (y_1'+2y_2') z_2,\\ &&
  v_2 +v_2'- (2y_1' +y_2') z_1 +(y_2'-y_1') z_2 \Big).
 \end{array}
 \end{equation}

We also need a discrete subgroup, which it could be taken to be
${\mathbb{Z}}^6 \subset G$. However, for later convenience, we
shall take the subgroup
 $$
 \Gamma = \{(y_1,y_2,z_1,z_2,v_1,v_2)\in {\mathbb{Z}}^6 \, |\, v_1 \equiv v_2
 \pmod 3 \} \subset G,
 $$
and define the nilmanifold
 $$
 N=\Gamma \backslash G\ .
 $$
In terms of a (global) system of coordinates
$(y_1,y_2,z_1,z_2,v_1,v_2)$ for $G$, the $1$--forms $\beta_i$,
$\gamma_i$ and $\eta_i$, $1\leq i \leq 2$, are  given by
 \begin{eqnarray*}
  \beta_i&=&dy_i, \quad  1\leq i \leq 2,  \quad \\
  \gamma_i &= &dz_i,  \quad  1\leq i \leq 2,  \quad \\
  \eta_1&=&dv_1-y_1 dz_1 +y_2dz_1+y_1 dz_2 +2y_2dz_2, \\
  \eta_2&=&dv_2+2y_1 dz_1 +y_2dz_1+y_1 dz_2 -y_2dz_2.
 \end{eqnarray*}
Note that $N$ is a principal torus bundle
 $$
 T^2={\mathbb{Z}}\langle (1,1),(3,0)\rangle
 \backslash {\mathbb{R}}^2 \hookrightarrow N \longrightarrow
 T^4= {\mathbb{Z}}^4\backslash {\mathbb{R}}^4,
 $$
with the projection $(y_1,y_2,z_1,z_2,v_1,v_2) \mapsto
(y_1,y_2,z_1,z_2)$.

\medskip

The Lie group $G$ can be also described as follows. Consider the
basis $\{\mu_i,\nu_i,\theta_i; 1\leq i \leq 2\}$ of the left
invariant $1$--forms on $G$ given by
 $$
 \begin{array}{ll}
 \mu_1 = \beta_1+\displaystyle\frac{1+\sqrt{3}}{2} \beta_2,  \quad &
 \mu_2 = \beta_1+ \displaystyle\frac{1-\sqrt{3}}{2} \beta_2,  \quad \\
 \nu_1 = \gamma_1+\displaystyle\frac{1+\sqrt{3}}{2} \gamma_2,  \quad &
 \nu_2 = \gamma_1+ \displaystyle\frac{1-\sqrt{3}}{2} \gamma_2,  \quad \\
 \theta_1 = \displaystyle\frac{2}{\sqrt{3}} \eta_1+
 \displaystyle\frac{1}{\sqrt{3}} \eta_2, \quad  & \theta_2=\eta_2.
 \end{array}
 $$
Hence, the structure equations can be rewritten as
 \begin{equation} \label{eqn:struc}
 \begin{array}{l}
  d\mu_i=0,  \quad  1\leq i \leq 2,  \\
  d\nu_i=0,  \quad  1\leq i \leq 2,  \\
  d\theta_1=\mu_1\wedge \nu_1- \mu_2\wedge \nu_2,  \\
  d\theta_2=\mu_1\wedge \nu_2+ \mu_2\wedge \nu_1.
 \end{array}
 \end{equation}
This means that $G$ is the complex Heisenberg group
$H_{{\mathbb{C}}}$, that is, the complex nilpotent Lie group of
complex matrices of the form
 $$
 \pmatrix{1&u_2&u_3\cr 0&1&u_1\cr 0&0&1\cr}.
 $$
In fact, in terms of the natural (complex) coordinate functions
$(u_1,u_2,u_3)$ on $H_{{\mathbb{C}}}$, we have that the complex
$1$--forms
 $$
 \mu=du_1, \ \nu=du_2, \ \theta=du_3-u_2 du_1
 $$
are left invariant and $d\mu=d\nu=0$, $d\theta=\mu\wedge\nu$. Now,
it is enough to take $\mu_1=\Re(\mu)$, $\mu_2=\Im(\mu)$,
$\nu_1=\Re(\nu)$, $\nu_2=\Im(\nu)$, $\theta_1=\Re(\theta)$,
$\theta_2=\Im(\theta)$ to recover equations (\ref{eqn:struc}),
where $\Re(\mu)$ and $\Im(\mu)$ denote the real and the imaginary
parts of $\mu$, respectively.

\begin{lemma} \label{lem:N}
  Let $\Lambda \subset {\mathbb{C}}$ be the lattice generated by $1$ and
  $\zeta=e^{2\pi i/3}$, and consider the discrete subgroup
  $\Gamma_H \subset H_{{\mathbb{C}}}$ formed
  by the matrices in which $u_1,u_2,u_3 \in \Lambda$. Then
  there is a natural identification of
  $N=\Gamma{\backslash}G$ with the quotient $\Gamma_H\backslash
  H_{{\mathbb{C}}}$.
\end{lemma}

\begin{proof}
  We have constructed above an isomorphism of Lie groups $G\to H_{{\mathbb{C}}}$, whose
  explicit equations are
   $$
   (y_1,y_2,z_1,z_2,v_1,v_2) \mapsto (u_1,u_2,u_3),
   $$
where
 \begin{eqnarray*}
 u_1 &=& \left( y_1+\displaystyle\frac{1+\sqrt{3}}{2} y_2  \right) +
 i \left( y_1+ \displaystyle\frac{1-\sqrt{3}}{2} y_2 \right), \\
 u_2 &=& \left( z_1+\displaystyle\frac{1+\sqrt{3}}{2} z_2  \right) +
 i \left( z_1+ \displaystyle\frac{1-\sqrt{3}}{2} z_2 \right), \\
 u_3 &=& \frac{1}{\sqrt{3}}\left( 2v_1+v_2 +3z_1y_2 +3z_2y_1+3z_2y_2
 \right) + i \left(v_2 + 2z_1y_1+z_2y_1 + z_1y_2 -z_2y_2\right).
 \end{eqnarray*}
Note that the formula for $u_3$ can be deduced from
 $$
 du_3-u_2du_1 =\theta =\left( \displaystyle\frac{2}{\sqrt{3}}\eta_1 +
 \displaystyle\frac{1}{\sqrt{3}}\eta_2 \right) + i\eta_2 \ .
 $$
Now the group $\Gamma\subset G$ corresponds under this isomorphism
to
 $$
 \left\{(u_1,u_2,u_3) | u_1, u_2 \in {\mathbb{Z}}\left\langle 1+i,
  \frac{1+\sqrt{3}}{2}+ \frac{1-\sqrt{3}}{2} i \right\rangle
  , u_3 \in {\mathbb{Z}}\left\langle
  2\sqrt{3}, \sqrt{3} + i \right\rangle \right\}.
  $$
Using the isomorphism of Lie groups  $H_{\mathbb{C}} \to
H_{\mathbb{C}}$ given by
 $$
 (u_1,u_2,u_3) \mapsto (u_1',u_2',u_3')=\left( \frac{u_1}{1+i},\frac{u_2}{1+i}
  ,\frac{u_3}{(1+i)^2}\right),
  $$
we get that $u_1',u_2', u_3'\in \Lambda={\mathbb{Z}}\langle
1,\zeta\rangle$, which completes the proof.
\end{proof}

\begin{remark}
 If we had considered the discrete subgroup ${\mathbb{Z}}^6\subset G$ instead of
 $\Gamma\subset G$, then we would not have obtained the fact $u_3'\in
 \Lambda$ in the proof of Lemma \ref{lem:N}.
Note that $N=\Gamma\backslash G \twoheadrightarrow
{\mathbb{Z}}^6\backslash G$ is a $3:1$ covering.
\end{remark}

Under the identification $N=\Gamma{\backslash}G \cong
\Gamma_H\backslash H_{{\mathbb{C}}}$, $N$ becomes the principal
torus bundle
 $$
 T^2={\Lambda \backslash {\mathbb{C}}} \hookrightarrow N \longrightarrow
 T^4= {\Lambda^2 \backslash {\mathbb{C}}^2},
 $$
with the projection $(u_1,u_2,u_3)\mapsto (u_1,u_2)$.

\section{A symplectic orbifold of dimension $8$} \label{sec:hatE}

We define the $8$--dimensional compact nilmanifold $M$ as  the
product
 $$
  M=T^2 \times N.
 $$
By Lemma \ref{lem:N} there is an isomorphism between $M$ and the
manifold $(\Gamma_H\backslash H_{{\mathbb{C}}}) \times (\Lambda
\backslash {\mathbb{C}})$ studied in \cite[Section 2]{FM4} (we
have to send the factor $T^2$ of $M$ to the factor $\Lambda
\backslash {\mathbb{C}}$). Clearly, $M$ is a principal torus
bundle
 $$
 T^2 \hookrightarrow M \stackrel{\pi}{\longrightarrow} T^6.
 $$
Let $(x_1,x_2)$ be the Lie algebra coordinates for $T^2$, so that
$(x_1,x_2,y_1,y_2,z_1,z_2,v_1,v_2)$ are coordinates for the Lie
algebra ${\mathbb{R}}^2\times G$ of $M$. Then
$\pi(x_1,x_2,y_1,y_2,z_1,z_2,v_1,v_2)=(x_1,x_2,y_1,y_2,z_1,z_2)$.
A basis for the left invariant (closed) $1$--forms on $T^2$ is
given as $\{\alpha_1,\alpha_2\}$, where $\alpha_1=dx_1$ and
$\alpha_2=dx_2$. Then $\{\alpha_i,\beta_i,\gamma_i,\eta_i; 1\leq i
\leq 2\}$ constitutes a (global) basis for the left invariant
$1$--forms on $M$. Note that $\{\alpha_i,\beta_i,\gamma_i; 1\leq i
\leq 2\}$ is a basis for the left invariant closed $1$--forms on
the base $T^6$. (We use the same notation for the differential
forms on $T^6$ and their pullbacks to $M$.) Using the computation
of the cohomology of $N$, we get that the Betti numbers of $M$
are: $b_0(M)=b_8(M)=1$, $b_1(M)=b_7(M)=6$, $b_2(M)=b_6(M)=17$,
$b_3(M)=b_5(M)=30$, $b_4(M)=36$. In particular, $\chi(M)=0$, as
for any nilmanifold.

Consider the action of the finite group ${\mathbb{Z}}_{3}$ on
${{\mathbb{R}}^2}$ given by
 $$
  \rho(x_1,x_2)=(-x_1-x_2,x_1),
 $$
for $(x_1,x_2)\in {{\mathbb{R}}^2}$, $\rho$ being the generator of
${\mathbb{Z}}_3$. Clearly
$\rho({\mathbb{Z}}^{2})={\mathbb{Z}}^{2}$, and so $\rho$ defines
an action of ${\mathbb{Z}}_{3}$ on the $2$-torus
$T^2={\mathbb{Z}}^2\backslash {\mathbb{R}}^2$ with $3$ fixed
points: $(0,0)$, $(\frac13,\frac13)$ and $(\frac23,\frac23)$. The
quotient space $T^2/{\mathbb{Z}}_{3}$ is the orbifold $2$--sphere
$S^2$ with $3$ points of multiplicity $3$. Let $x_1$, $x_2$ denote
the natural coordinate functions on ${{\mathbb{R}}^2}$. Then the
$1$--forms $dx_1$, $dx_2$ satisfy $\rho^*(dx_1)=-dx_1-dx_2$ and
$\rho^*(dx_2)=dx_1$, hence $\rho^*(-dx_1-dx_2)=dx_2$. Thus, we can
take the $1$--forms $\alpha_1$ and $\alpha_2$ on $T^2$ such that
 \begin{equation} \label{eqn:v1}
   \rho^*(\alpha_1)=-\alpha_1-\alpha_2, \quad
   \rho^*(\alpha_2)=\alpha_1.
 \end{equation}

Define the following action of ${\mathbb{Z}}_3$ on $M$, given, at
the level of Lie groups, by $\rho \colon {{\mathbb{R}}^2} \times
{{\mathbb{R}}^6}\longrightarrow {{\mathbb{R}}^2} \times
{{\mathbb{R}}^6}$,
 $$
 \rho(x_1,x_2,y_1,y_2,z_1,z_2,v_1,v_2)
 =(-x_1-x_2,x_1,-y_1-y_2,y_1,-z_1-z_2,z_1,-v_1-v_2,v_1).
 $$
Note that $m(\rho(p'),\rho(p)) = \rho (m (p',p))$, for all
$p,p'\in G$, where $m$ is the multiplication map (\ref{eqn:m}) for
$G$. Also $\Gamma\subset G$ is stable by $\rho$ since
 $$
 v_1\equiv v_2 \pmod 3 \Longrightarrow -v_1-v_2\equiv v_1 \pmod 3.
 $$
Therefore there is a induced map $\rho \colon M \to M$, and this
covers the action $\rho: T^6 \to T^6$ on the $6$--torus
$T^{6}=T^{2}\times T^{2}\times T^{2}$ (defined as the action
$\rho$ on each of the three factors simultaneously). The action of
$\rho$ on the fiber $T^{2}= {\mathbb{Z}}\langle
(1,1),(3,0)\rangle$ has also $3$ fixed points: $(0,0)$, $(1,0)$
and $(2,0)$. Hence there are $3^4=81$ fixed points on $M$.

\begin{remark}
Under the isomorphism $M \cong (\Gamma_H\backslash
H_{{\mathbb{C}}}) \times (\Lambda \backslash {\mathbb{C}})$,
 we have that the action of $\rho$ becomes $\rho(u_1,u_2,u_3)=(\bar \zeta
 u_1,\bar \zeta u_2,\zeta u_3)$, where
 $\zeta=e^{2\pi i/3}$. Composing the isomorphism of Lemma
 \ref{lem:N} with the conjugation $(u_1,u_2,u_3)\mapsto
 (v_1,v_2,v_3)=(\bar{u}_1,\bar{u}_2,
 \bar{u}_3)$ (which is an isomorphism of Lie groups $H_{\mathbb{C}} \to H_{\mathbb{C}}$ leaving $\Gamma_H$
 invariant),
 we have that the action of $\rho$ becomes
 $\rho(v_1,v_2,v_3)=(\zeta v_1, \zeta v_2,\zeta^2 v_3)$. This is
 the action used in \cite{FM4}.
\end{remark}

We take the basis $\{\alpha_i,\beta_i,\gamma_i,\eta_i; 1\leq i
\leq 2\}$ of the $1$--forms on $M$ considered above. The
$1$--forms $dy_i$, $dz_i$, $dv_i$, $1\leq i\leq 2$, on $G$ satisfy
the following conditions similar to (\ref{eqn:v1}):
$\rho^*(dy_1)=-dy_1-dy_2$, $\rho^*(dy_2)=dy_1$,
$\rho^*(dz_1)=-dz_1-dz_2$, $\rho^*(dz_2)=dz_1$,
$\rho^*(dv_1)=-dv_1-dv_2$, $\rho^*(dv_2)=dv_1$. So
 \begin{equation} \label{eqn:v2}
 \begin{array}{ll}
  \rho^*(\alpha_1)=-\alpha_1-\alpha_2, \quad   & \rho^*(\alpha_2)=\alpha_1,\\
  \rho^*(\beta_1)=-\beta_1-\beta_2, \quad  & \rho^*(\beta_2)=\beta_1,\\
  \rho^*(\gamma_1)=-\gamma_1-\gamma_2,  \quad  &\rho^*(\gamma_2)=\gamma_1,\\
  \rho^*(\eta_1)=-\eta_1-\eta_2,  \quad &\rho^*(\eta_2)=\eta_1.
 \end{array}
 \end{equation}

\begin{remark} \label{rem:3}
If we define the $1$--forms $\alpha_3=-\alpha_1-\alpha_2$,
$\beta_3=-\beta_1-\beta_2$, $\gamma_3=-\gamma_1-\gamma_2$ and
$\eta_3=-\eta_1-\eta_2$, then we have
$\rho^*(\alpha_1)=\alpha_3,\, \rho^*(\alpha_2)=\alpha_1,\,
\rho^*(\alpha_3)=\alpha_2$, and analogously for the others.
\end{remark}

Define the quotient space
 $$
 \widehat{M}=M/{\mathbb{Z}}_{3},
 $$
and denote by $\varphi:M\to \widehat{M}$ the projection. It is an
orbifold, and it admits the structure of a symplectic orbifold
(see \cite{CMF} for a general discussion on symplectic orbifolds).

\begin{proposition}\label{prop:symplectic}
The $2$--form $\omega$ on $M$ defined by
 $$
  \omega=\alpha_1\wedge\alpha_2+\eta_2\wedge\beta_1-
  \eta_1\wedge\beta_2+\gamma_1\wedge \gamma_2
 $$
is a ${\mathbb{Z}}_{3}$-invariant symplectic form on $M$.
Therefore it induces $\widehat\omega\in
\Omega^2_{\mathrm{orb}}(\widehat{M})$, such that
$(\widehat{M},\widehat\omega)$ is a symplectic orbifold.
\end{proposition}

\begin{proof}
Clearly $\omega^4\not=0$. Using (\ref{eqn:v2}) we have that
$\rho^*(\omega)=(-\alpha_1-\alpha_2)\wedge\alpha_1+
\eta_1\wedge(-\beta_1-\beta_2)+(\eta_1+\eta_2)\wedge\beta_1+
(-\gamma_1-\gamma_2)\wedge\gamma_1 =\omega$, so $\omega$ is
${\mathbb{Z}}_{3}$-invariant. Finally,
 $$
 d\omega= d\eta_2\wedge\beta_1- d\eta_1\wedge\beta_2
 =(\beta_2\wedge\gamma_1-\beta_2\wedge\gamma_2)\wedge\beta_1
 -(-\beta_1\wedge\gamma_1+\beta_1\wedge\gamma_2)\wedge\beta_2=
 0.
 $$
\end{proof}

It can be seen (cf.\ proof of Proposition 2.3 in \cite{FM4}) that
$\widehat{M}$ is simply connected. Moreover, its cohomology can be
computed using that
 $$
 H^*(\widehat{M})= H^*(M)^{{\mathbb{Z}}_3} \, .
 $$
We get
 $$
  \begin{array}{lcl}
  H^1(\widehat{M}) &=& 0, \\
  H^2(\widehat{M}) &=& \langle [\alpha_1 \wedge \alpha_2],
  [\alpha_1 \wedge \beta_2-\alpha_2 \wedge \beta_1],
    [\alpha_1 \wedge \beta_1+\alpha_1 \wedge \beta_2+\alpha_2 \wedge \beta_2],
    \\
   &&[\alpha_1 \wedge \gamma_2-\alpha_2 \wedge \gamma_1],
    [\alpha_1 \wedge \gamma_1+\alpha_1 \wedge \gamma_2+\alpha_2 \wedge \gamma_2],
   [\beta_1 \wedge \beta_2],
     [\beta_1 \wedge \gamma_2-\beta_2 \wedge \gamma_1],
  \\
 && [\beta_1 \wedge\gamma_1+\beta_1 \wedge\gamma_2+\beta_2 \wedge\gamma_2],
 [\beta_1\wedge \eta_2 - \beta_2\wedge \eta_1],
 [\beta_1\wedge \eta_1 + \beta_1\wedge \eta_2+\beta_2\wedge \eta_2],
 \\
 &&  [\gamma_1\wedge \gamma_2],[\gamma_1\wedge \eta_2 - \gamma_2\wedge \eta_1],
  [\gamma_1\wedge \eta_1 + \gamma_1\wedge \eta_2+\gamma_2\wedge \eta_2]\rangle,\\
  H^3(\widehat{M}) &=& 0.
  \end{array}
  $$

\begin{remark}
The Euler characteristic of $\widehat{M}$ can be computed via the
formula for finite group action quotients: let $\Pi$ be the cyclic
group of order $n$, acting on a space $X$ almost freely. Then
 $$
 \chi(X/\Pi)= \frac1n \chi(X) + \sum_{p} \left(1- \frac{1}{\#
 \Pi_p}\right),
 $$
where $\Pi_p\subset \Pi$ is the isotropy group of $p\in X$. In our
case $\chi(\widehat{M})= \frac13 \chi (M) + 81 (1-\frac13) =54$.
\end{remark}

Using this remark and the previous calculation, we get that
$b_1(\widehat{M})=b_7(\widehat{M})=0$, $b_2(\widehat{M}) =
b_6(\widehat{M}) ={13}$, $b_3(\widehat{M}) =b_5(\widehat{M}) =0$
and $b_4(\widehat{M}) =26$. Note that $\widehat{M}$ satisfies
Poincar{\'e} duality since $H^*(\widehat{M})= H^*(M)^{{\mathbb{Z}}_3}$
and $H^*(M)$ satisfies Poincar{\'e} duality.

\section{Non-formality of the symplectic orbifold} \label{sec:non-formal}

Formality is a property of the rational homotopy type of a space
which is of great importance in symplectic geometry. This is due
to the fact that compact K\"ahler manifolds are formal \cite{DGMS}
whilst there are compact symplectic manifolds which are non-formal
\cite{TO,BT,FM4}. A general discussion of the property of
formality can be found in \cite{TO}.

The non-formality of a space can be detected by means of Massey
products. Let us recall its definition. The simplest type of
Massey product is the triple (also known as ordinary) Massey
product. Let $X$ be a smooth manifold and let $a_i \in
H^{p_i}(X)$, $1 \leq i\leq 3$, be three cohomology classes such
that $a_1\cup a_2=0$ and $a_2\cup a_3=0$. The (triple) Massey
product of the classes $a_i$ is defined as the set
  $$
  \langle a_1,a_2,a_3 \rangle  = \{
  [ \alpha_1 \wedge \eta+(-1)^{p_1+1} \xi
  \wedge \alpha_3] \ | \ a_i=[\alpha_i],\ \alpha_1\wedge \alpha_2= d \xi,
  \ \alpha_2\wedge \alpha_3=d \eta \}
  $$
inside $H^{p_1+p_2+p_3-1}(X)$. We say that $\langle a_1,a_2,a_3
\rangle$ is trivial if $0\in \langle a_1,a_2,a_3 \rangle$.

The definition of higher Massey products is as follows
(see~\cite{Mas,TO}). The Massey product $\langle
a_1,a_2,\dots,a_t\rangle$, $a_i\in H^{p_i}(X)$, $1\leq i\leq t$,
$t\geq 3$, is defined if there are differential forms
$\alpha_{i,j}$ on $X$, with $1\leq i\leq j\leq t$, except for the
case $(i,j)=(1,t)$, such that
 \begin{equation}\label{eqn:gm}
 a_i=[\alpha_{i,i}], \qquad
 d\,\alpha_{i,j}= \sum\limits_{k=i}^{j-1} {\bar \alpha}_{i,k}\wedge
 \alpha_{k+1,j},
 \end{equation}
where $\bar \alpha=(-1)^{\deg(\alpha)} \alpha$. Then the Massey
product is
 $$
 \langle a_1,a_2,\dots,a_t \rangle =\left\{
 \left[\sum\limits_{k=1}^{t-1} {\bar \alpha}_{1,k} \wedge
 \alpha_{k+1,t}\right] \ | \ \alpha_{i,j} \hbox{ as in (\ref{eqn:gm})}\right\}
 \subset H^{p_1+ \cdots +p_t
 -(t-2)}(X)\, .
 $$
We say that the Massey product is trivial if $0\in \langle
a_1,a_2,\dots,a_t\rangle$. Note that for $\langle
a_1,a_2,\dots,a_t\rangle$ to be defined it is necessary that
$\langle a_1,\dots,a_{t-1}\rangle$ and $\langle
a_2,\dots,a_t\rangle$ are defined and trivial.

The existence of a non-trivial Massey product is an obstruction to
formality, namely, if $X$ has a non-trivial Massey product then
$X$ is non-formal.

\medskip

In the case of an orbifold, Massey products are defined
analogously but taking the forms to be \emph{orbifold forms} (see
\cite[Section 2]{CMF}).

Now we want to prove the non-formality of the orbifold
$\widehat{M}$ constructed in the previous section. By the results
of \cite{TO}, $M$ is non-formal since it is a nilmanifold which is
not a torus. We shall see that this property is inherited by the
quotient space $\widehat{M}=M/{\mathbb{Z}}_3$. For this, we study
the Massey products on $\widehat{M}$.

\begin{lemma}\label{lem:MasseyhatE}
 $\widehat{M}$ has a non-trivial Massey product  if and only if  $M$
 has a non-trivial Massey product with all cohomology classes
 $a_i\in H^*(M)$ being ${\mathbb{Z}}_3$-invariant cohomology classes.
\end{lemma}

\begin{proof}
We shall do the case of triple Massey products, since the general
case is similar. Suppose that $\langle a_1,a_2,a_3\rangle$,  $a_i
\in H^{p_i}(\widehat{M})$, $1 \leq i\leq 3$ is a non-trivial
Massey product on $\widehat{M}$. Let $a_i=[\alpha_i]$, where
$\alpha_i\in \Omega^*_{\mathrm{orb}}(\widehat{M})$. We pull-back
the cohomology classes $\alpha_i$ via
$\varphi^*:\Omega^*_{\mathrm{orb}}(\widehat{M})\to \Omega^*(M)$ to
get a Massey product $\langle
[\varphi^*\alpha_1],[\varphi^*\alpha_2],[\varphi^*\alpha_3]\rangle$.
Suppose that this is trivial on $M$, then $\varphi^*\alpha_1\wedge
\varphi^*\alpha_2= d \xi$, $\varphi^*\alpha_2\wedge
\varphi^*\alpha_3=d \eta$, with $\xi,\eta\in \Omega^*(M)$, and
$\varphi^*\alpha_1 \wedge \eta+(-1)^{p_1+1} \xi \wedge
\varphi^*\alpha_3= df$. Then $\tilde\eta=
(\eta+\rho^*\eta+(\rho^*)^2\eta)/3$, $\tilde\xi=
(\xi+\rho^*\xi+(\rho^*)^2\xi)/3$ and $\tilde{f}= (f+
\rho^*\eta+(\rho^*)^2\eta)/3$ are ${\mathbb{Z}}_3$-invariant and
$\varphi^*\alpha_1 \wedge \tilde\eta+(-1)^{p_1+1} \tilde\xi \wedge
\varphi^*\alpha_3= d\tilde{f}$. Writing
$\tilde\eta=\phi^*\hat\eta$, $\tilde\xi=\phi^*\hat\xi$, $\tilde
f=\phi^*\hat f$, for $\hat\eta,\hat\xi,\hat f\in
\Omega^*_{\mathrm{orb}}(\widehat M)$, we get $\alpha_1 \wedge \hat
\eta+(-1)^{p_1+1} \hat\xi \wedge \alpha_3= d\hat f$, contradicting
that $\langle a_1,a_2,a_3\rangle$ is non-trivial.

Conversely, suppose that $\langle a_1,a_2,a_3\rangle$, $a_i \in
H^{p_i}(M)^{{\mathbb{Z}}_3}$, $1 \leq i\leq 3$, is a non-trivial
Massey product on $M$. Then we can represent $a_i=[\alpha_i]$ by
${\mathbb{Z}}_3$-invariant differential forms $\alpha_i\in
\Omega^{p_i}(M)$. Let $\hat\alpha_i$ be the induced form on
$\widehat{M}$. Then $\langle
[\hat\alpha_1],[\hat\alpha_2],[\hat\alpha_3]\rangle$ is a
non-trivial Massey product on $\widehat{M}$. For if it were
trivial then pulling-back by $\varphi$, we would get $0\in \langle
\varphi^*[\hat\alpha_1], \varphi^*[\hat\alpha_2],
\varphi^*[\hat\alpha_3]\rangle=\langle a_1,a_2,a_3\rangle$.
\end{proof}

In our case, all the triple and quintuple Massey products on
$\widehat{M}$ are trivial. For instance, for a Massey product of
the form $\langle a_1,a_2,a_3\rangle$, all $a_i$ should have even
degree, since
$H^1(\widehat{M})=H^3(\widehat{M})=H^5(\widehat{M})=H^7(\widehat{M})=0$.
Therefore the degree of the cohomology classes in $\langle
a_1,a_2,a_3\rangle$ is odd, hence they are zero.

Since the dimension of $\widehat{M}$ is $8$, there is no room for
sextuple Massey products or higher, since the degree of $\langle
a_1,a_2,\ldots, a_s\rangle$ is at least $s+2$, as $\deg a_i\geq
2$. For $s=6$, a sextuple Massey product of cohomology classes of
degree $2$ would live in the top degree cohomology. For computing
an element of $\langle a_1,\ldots, a_6\rangle$, we have to choose
$\alpha_{i,j}$ in (\ref{eqn:gm}). But then adding a closed form
$\phi$ with $a_1 \cup [\phi]=\lambda [\widehat{M}]\in
H^8(\widehat{M})$ to $\alpha_{2,6}$ we can get another element of
$\langle a_1,\ldots, a_6\rangle$ which is the previous one plus
$\lambda [\widehat{M}]$. For suitable $\lambda$ the we get $0\in
\langle a_1,\ldots, a_6\rangle$.

The only possibility for checking the non-formality of
$\widehat{M}$ via Massey products is to get a non-trivial
quadruple Massey product.

{}From now on, we will denote by the same symbol a
${\mathbb{Z}}_3$-invariant form on $M$ and that induced on
$\widehat{M}$. Notice that the $2$ forms $\gamma_1\wedge\gamma_2$,
$\beta_1\wedge\beta_2$ and
$\alpha_1\wedge\gamma_1+\alpha_2\wedge\gamma_1+\alpha_2\wedge\gamma_2$
are ${\mathbb{Z}}_3$-invariant forms on $M$, hence they descend to
the quotient $\widehat{M}=M/{\mathbb{Z}}_3$. We have the
following:

\begin{proposition}\label{prop:nonformal1}
 The quadruple Massey product
 $$
 \langle [\gamma_1\wedge\gamma_2],[\beta_1\wedge\beta_2],
 [\beta_1\wedge\beta_2],
 [\alpha_1\wedge\gamma_1+\alpha_2\wedge\gamma_1+\alpha_2\wedge\gamma_2]\rangle
 $$
 is non-trivial on $\widehat{M}$.
 Therefore, the space $\widehat{M}$ is non-formal.
\end{proposition}

\begin{proof}
First we see that
 \begin{eqnarray*}
(\gamma_1\wedge\gamma_2)\wedge (\beta_1\wedge\beta_2)&=&d\xi, \\
(\beta_1\wedge\beta_2)\wedge
(\alpha_1\wedge\gamma_1+\alpha_2\wedge\gamma_1+\alpha_2\wedge\gamma_2)&=&
d\varsigma,
 \end{eqnarray*}
 where $\xi$ and $\varsigma$ are the differential
$3$--forms on $\widehat{M}$ given by
  \begin{eqnarray*}
  \xi &=& - \frac{1}{6} \left( \gamma_1\wedge (\beta_1\wedge\eta_2+
   \beta_2\wedge\eta_2+\beta_2\wedge\eta_1 ) +\gamma_2\wedge(
   \beta_1\wedge\eta_2 + \beta_1\wedge\eta_1+ \beta_2\wedge\eta_1)\right), \\
  \varsigma &=&
  \frac{1}{3}\left(- \alpha_1\wedge(\eta_2\wedge\beta_1+
  \eta_1\wedge\beta_1+\eta_1\wedge\beta_2)+
  \alpha_2\wedge(\eta_2\wedge\beta_2-\eta_1\wedge\beta_1)\right).
 \end{eqnarray*}
Therefore, the triple Massey products $\langle
[\gamma_1\wedge\gamma_2],[\beta_1\wedge\beta_2],
[\beta_1\wedge\beta_2]\rangle$ and $\langle
[\beta_1\wedge\beta_2], [\beta_1\wedge\beta_2],
[\alpha_1\wedge\gamma_1+\alpha_2\wedge\gamma_1+\alpha_2\wedge\gamma_2]\rangle$
are defined, and they are trivial because all the (triple) Massey
products on $\widehat{M}$ are trivial. (Notice that the forms
$\xi$ and $\varsigma$ are ${\mathbb{Z}}_3$-invariant on $M$ and so
descend to $\widehat{M}$.) Therefore,  the quadruple Massey
product $\langle [\gamma_1\wedge\gamma_2],[\beta_1\wedge\beta_2],
[\beta_1\wedge\beta_2],
[\alpha_1\wedge\gamma_1+\alpha_2\wedge\gamma_1+\alpha_2\wedge\gamma_2]\rangle$
is defined on $\widehat{M}$. Moreover, it is trivial on
$\widehat{M}$ if and only if there are differential forms $f_i \in
\Omega^3(\widehat{M})$, $1\leq i\leq 3$, and $g_j \in
\Omega^4(\widehat{M})$, $1\leq j\leq 2$, such that
  \begin{eqnarray*}
  && (\gamma_1\wedge\gamma_2)\wedge (\beta_1\wedge\beta_2)=d(\xi+f_1), \\
  && (\beta_1\wedge\beta_2)\wedge (\beta_1\wedge\beta_2)=d f_2, \\
  && (\beta_1\wedge\beta_2)\wedge
   (\alpha_1\wedge\gamma_1+\alpha_2\wedge\gamma_1+\alpha_2\wedge\gamma_2)=
   d (\varsigma+f_3),  \\
  && (\gamma_1\wedge\gamma_2) \wedge f_2-(\xi+f_1)\wedge
   (\beta_1\wedge\beta_2)=d g_1,  \\
  && (\beta_1\wedge\beta_2)\wedge (\varsigma +f_3)-
   f_2\wedge (\alpha_1\wedge\gamma_1+
   \alpha_2\wedge\gamma_1+\alpha_2\wedge\gamma_2)=d g_2,
  \end{eqnarray*}
and the $6$--form given by
 $$
 \Psi=-(\gamma_1\wedge\gamma_2)\wedge g_2 - g_1\wedge
  (\alpha_1\wedge\gamma_1+\alpha_2\wedge\gamma_1+\alpha_2\wedge\gamma_2)
  +(\xi+f_1)\wedge(\varsigma +f_3)
 $$
defines the zero class in $H^6(\widehat{M})$. Clearly $f_1$, $f_2$
and $f_3$ are closed $3$--forms. Since $H^3(\widehat{M})=0$, we
can write $f_1=df'_1$, $f_2=df'_2$ and $f_3=df'_3$  for some
differential $2$--forms $f'_1$, $f'_2$  and $f'_3\in
\Omega^2(\widehat{M})$. Now, multiplying $[\Psi]$ by the
cohomology class $[\sigma]\in H^2(\widehat{M})$, where
$\sigma=2\alpha_1\wedge \gamma_2 -\alpha_2\wedge\gamma_1
+\alpha_1\wedge\gamma_1+\alpha_2\wedge\gamma_2$ we get
 $$
   \sigma \wedge \Psi
 =-\frac{1}{3}(\alpha_1\wedge\alpha_2\wedge\beta_1\wedge\beta_2
  \wedge\gamma_1\wedge\gamma_2\wedge\eta_1\wedge\eta_2)
   + d( \sigma\wedge\xi\wedge f_3' + \sigma \wedge \varsigma \wedge
 f_1'+\sigma \wedge  f_1' \wedge df_3').
 $$
Hence, $[2\alpha_1\wedge \gamma_2 -\alpha_2\wedge\gamma_1
+\alpha_1\wedge\gamma_1+\alpha_2\wedge\gamma_2]\cup [\Psi] \neq
0$, which implies that $[\Psi]$ is non-zero in $H^6(\widehat{M})$.
This proves that the Massey product $\langle
[\gamma_1\wedge\gamma_2],[\beta_1\wedge\beta_2],
[\beta_1\wedge\beta_2],
[\alpha_1\wedge\gamma_1+\alpha_2\wedge\gamma_1+\alpha_2\wedge\gamma_2]\rangle$
is non-trivial, and so $\widehat{M}$ is non-formal.
\end{proof}


Finally, 
there is a way to desingularize $(\widehat{M},\widehat{\omega})$
to get a smooth symplectic 
manifold.

\begin{theorem}
 There is a smooth compact
 symplectic $8$-manifold $(\widetilde{M},\widetilde{\omega})$
 which is simply-connected and non-formal.
\end{theorem}

\begin{proof}
By \cite[Theorem 3.3]{CMF}, there is a symplectic resolution
$\pi:(\widetilde{M},\widetilde{\omega})\to
(\widehat{M},\widehat{\omega})$, which consists of a smooth
symplectic manifold $(\widetilde{M},\widetilde{\omega})$ and a map
$\pi$ which is a diffeomorphism outside the singular points.

To prove the non-formality of $\widetilde{M}$, we work as follows.
All the forms of the proof of Proposition \ref{prop:nonformal1}
can be defined on the resolution $\widetilde{M}$. Take a
${\mathbb{Z}}_3$-equivariant map $\psi:M\to M$ which is the
identity outside small balls around the fixed points, and
contracts smaller balls onto the fixed points. Substitute the
forms $\vartheta$, $\tau_i$, $\kappa$, $\xi$, \ldots\ by
$\psi^*\vartheta$, $\psi^*\tau_i$, $\psi^*\kappa$, $\psi^*\xi$,
\ldots\ Then the corresponding elements in the quadruple Massey
product are non-zero, but these forms are zero in a neighbourhood
of the fixed points. Therefore they define forms on
$\widetilde{M}$, by extending them by zero along the exceptional
divisors $E_p=\pi^{-1}(p)$ ($p\in\widehat{M}$ singular point). Now
the proof of Proposition \ref{prop:nonformal1} works for
$\widetilde{M}$ with these forms.

Finally, the manifold $\widetilde{M}$ is simply connected as it is
proved in \cite[Proposition 2.3]{FM4} (basically, this follows
from the simply-connectivity of $\widehat{M}$).
\end{proof}

\begin{theacknowledgments}
First author partially supported by EPSRC (UK) grant. Second
author partially supported by grant MCyT (Spain)
MTM2005-08757-C04-02. Third author partially supported by grant
MCyT (Spain) MTM2004-07090-C03-01.
\end{theacknowledgments}


\begin{thebibliography}{33}

\bibitem{ref1} I.K.~Babenko, and I.A.~Taimanov, On non-formal simply connected
symplectic manifolds, \emph{Siberian Math. J.} \textbf{41},
204--217 (2000).

\bibitem{ref2} I.K.~Babenko, and I.A.~Taimanov, Massey products in
symplectic manifolds, \emph{Sb. Math.} \textbf{191}, 1107--1146
(2000).

\bibitem{BT} I.K.~Babenko, and I.A.~Taimanov, On existence of non-formal simply connected
symplectic manifolds, \emph{Russian Math. Surveys} \textbf{53},
1082--1083 (1998).

\bibitem{CMF} G.~Cavalcanti, M.~Fern\'andez, and V.~Mu\~noz,
Symplectic resolutions, Lefschetz property and formality,
\emph{Advances in Math.} To appear.

\bibitem{DGMS}
P.~Deligne, P.~Griffiths, J.~Morgan, and D.~Sullivan,
Real homotopy theory of K\"ahler manifolds,
\emph{Invent. Math.} \textbf{29}, 245--274 (1975).

\bibitem{FM4}
M.~Fern\'{a}ndez, and V.~Mu\~{n}oz, An $8$-dimensional non-formal
simply connected symplectic manifold, \emph{Annals of Math. (2)}.
To appear.

\bibitem{Malc}
A.I.~Mal'cev, A class of homogeneous spaces, \emph{Amer. Math.
Soc. Transl.} \textbf{39} (1951).

\bibitem{Mas}
W.S.~Massey, Some higher order cohomology operations, \emph{Int.
Symp. Alg. Top. Mexico}, 145--154 (1958).

\bibitem{No}
K.~Nomizu, On the cohomology of compact homogeneous spaces of
nilpotent Lie groups, \emph{Annals of Math. (2)} \textbf{59},
531--538 (1954).

\bibitem{ref3} Y.~Rudyak, and A.~Tralle, Thom spaces, Massey products and
nonformal symplectic manifolds, \emph{Internat. Math. Res.
Notices} \textbf{10}, 495--513 (2000).

\bibitem{TO} A.~Tralle, and J.~Oprea, \emph{Symplectic manifolds with no K\"ahler
structure}, Lecture Notes in Math. \textbf{1661},
Springer--Verlag, 1997.


\end{thebibliography}
\end{document}